\documentclass[12pt]{amsart}
\usepackage{amssymb}
\usepackage[all]{xy}
\usepackage{mathrsfs}
\newtheorem{theorem}{Theorem}[section]
\newtheorem{lemma}[theorem]{Lemma}
\newtheorem{proposition}[theorem]{Proposition}

\newtheorem{remark}[theorem]{Remark}
\newtheorem{example}[theorem]{Example}
\newtheorem{corollary}[theorem]{Corollary}
\def\to{\rightarrow}
\def\r{\mathrm}
\def\c{\mathcal}
\def\b{\mathbf}

\def\f{\mathfrak}
\def\ot{\otimes}
\begin{document}
\baselineskip15pt
\title[Banach algebras of generalized matrices]
{A class of Banach algebras of generalized matrices}
\author[M.M. Sadr]{Maysam Maysami Sadr}
\address{Depertment of Mathematics\\
Institute for Advanced Studies in Basic Sciences (IASBS)\\
P.O. Box 45195-1159, Zanjan 45137-66731, Iran}
\email{sadr@iasbs.ac.ir}
\begin{abstract}
We introduce a class of Banach algebras of generalized matrices and study the existence of approximate units, ideal structure,
and derivations of them.
\end{abstract}
\maketitle
\section{Introduction}
Let $\f{X}$ be a compact metrizable space and $\f{m}$ be a Borel probability measure on $\f{X}$. In this note we study some aspects of the
algebraic structure of a Banach algebra $\c{M}$ of generalized complex matrices whose their arrays are indexed by elements of $\f{X}^2$ and
vary continuously. The multiplication of $\c{M}$ is defined similar to the ordinary matrix multiplication and uses $\f{m}$
as the weight for arrays. See Section 2 for exact definition. In the case that $\f{m}$ has full support, $\c{M}$ is isometric isomorphic to
a subalgebra of compact operators acting on the Banach space of continuous functions on $\f{X}$. Indeed any element of $\c{M}$
defines an integral operator in a canonical way. Thus $\c{M}$ can be interpreted as a Banach algebra of integral operators or kernels (\cite{Beaver1}).
In Section 3 we investigate the existence of approximate units of $\c{M}$. In Section 4 we show that if $\f{X}$ is infinite then the center of $\c{M}$
is zero. In Section 5 we study ideal structure of $\c{M}$. In Section 6 we consider some classes of representations of $\c{M}$.
In Section 7 we show that under some mild conditions bounded derivations on $\c{M}$ are approximately inner.

\textbf{Notations.}
For a compact space $\f{X}$ and a Banach space $E$ we denote by $\b{C}(\f{X};E)$ the Banach space of continuous $E$-valued functions on $\f{X}$
with supremum norm. We also let $\b{C}(\f{X}):=\b{C}(\f{X};\mathbb{C})$. There is a canonical isometric isomorphism
$\b{C}(\f{X};E)\cong\b{C}(\f{X})\check{\ot} E$ where $\check{\ot}$ denotes the completed injective tensor product.
The phrase ``point-wise convergence topology''  is abbreviated to ``pct''.
By pct on $\b{C}(\f{X};E)$ we mean the vector topology under which a net $(f_\lambda)_\lambda\in\b{C}(\f{X};E)$ converges to $f$ if and only if
$f_\lambda(x)\to f(x)$ in the norm of $E$ for every $x\in\f{X}$. If $f$ and $f'$ are complex functions on spaces $\f{X}$ and $\f{X}'$
then $f\ot f'$ denotes the function on $\f{X}\times\f{X}'$ defined by $(x,x')\mapsto f(x)f'(x')$.
The support of a Borel measure $\f{m}$ is denoted by $\r{Sp}\f{m}$. $\c{B}_{x,\delta}$ denotes the open ball with center at $x$ and radius $\delta$.
\section{The main definitions}
Let $\f{X}$ be a compact metrizable space and $\f{m}$ be a Borel probability measure on $\f{X}$.
By analogy with matrix multiplication we let the convolution of $f,g\in\b{C}(\f{X}^2)$ be defined by
$f\star g(x,y)=\int_Xf(x,z)g(z,y)\r{d}\f{m}(z)$. Also by analogy with matrix adjoint we let $f^*\in\b{C}(\f{X}^2)$ be defined by
$f^*(x,y)=\bar{f}(y,x)$. It is easily verified that $\star$ is an associative multiplication, $*$ is an involution, and also,
$\|f\star g\|_\infty\leq\|f\|_\infty\|g\|_\infty$ and $\|f^*\|_\infty=\|f\|_\infty$; thus $\b{C}(\f{X}^2)$ becomes a Banach $*$-algebra
which we denote by $\c{M}_{\f{X},\f{m}}$. If $\f{X}$ is a finite space with $n$ distinct elements
$x_1,\cdots,x_n$ and $\r{Sp}\f{m}=\f{X}$ then the assignment $(a_{ij})\mapsto((x_i,x_j)\mapsto\frac{1}{\f{m}\{x_i\}}a_{ij})$
defines a $*$-algebra isomorphism from the algebra of $n\times n$ matrices onto $\c{M}_{\f{X},\f{m}}$.

Beside norm and pc topologies on $\c{M}_{\f{X},\f{m}}$ we need two other topologies:
Consider the canonical isometric isomorphism $f\mapsto(y\mapsto f(\cdot,y))$ from $\b{C}(\f{X}^2)$ onto $\b{C}(\f{X};\b{C}(\f{X}))$.
We define the column-wise convergence topology (cct for short)
on $\c{M}_{\f{X},\f{m}}$ to be the pull back of the pct on $\b{C}(\f{X};\b{C}(\f{X}))$ under this isomorphism.
The row-wise convergence topology (rct for short) on $\c{M}_{\f{X},\f{m}}$ is defined similarly
by using the other canonical isomorphism  $f\mapsto(x\mapsto f(x,{\cdot}))$.
The pct on $\b{C}(\f{X}^2)=\c{M}_{\f{X},\f{m}}$ is contained in the intersection of cct and rct.
Column-wise and row-wise cts are \emph{adjoint} to each other in the sense that the involution
$*$ from $\c{M}_{\f{X},\f{m}}$ with cct to $\c{M}_{\f{X},\f{m}}$ with rct is a homeomorphism.
If $a_\lambda\xrightarrow{cct}a$ and $b_\lambda\xrightarrow{rct}b$ in $\c{M}_{\f{X},\f{m}}$, then
$c\star a_\lambda\xrightarrow{cct}c\star a$ and $b_\lambda\star c\xrightarrow{rct}b\star c$ for every $c$.

The assignment $(\f{X},\f{m})\mapsto\c{M}_{\f{X},\f{m}}$ can be considered as a cofunctor from the category of pairs $(\f{X},\f{m})$
to the category of Banach $*$-algebras: Suppose that $(\f{X}',\f{m}')$ is another pair of a compact metrizable space
and a Borel probability measure on it. Let $\alpha:\f{X}'\to\f{X}$
be a measure preserving continuous map. Then $\alpha$ induces a bounded $*$-algebra morphism $\c{M}\alpha$ from  $\c{M}_{\f{X},\f{m}}$ into
$\c{M}_{\f{X}',\f{m}'}$ defined by $[(\c{M}\alpha)f](x',y')=f(\alpha(x'),\alpha(y'))$. By an explicit example we show that
$\c{M}$ as a functor is not full: Let $\beta:\f{X}\to\mathbb{C}$ be a continuous function with $|\beta|=1_\f{X}$ and $\beta\neq 1_\f{X}$.
Then $\hat{\beta}:\c{M}_{\f{X},\f{m}}\to\c{M}_{\f{X},\f{m}}$ defined by $(\hat{\beta}f)(x,y)=\beta(x)f(x,y)\bar{\beta}(y)$
is an isometric $*$-algebra isomorphism. It is clear that $\hat{\beta}$ is not of the form $\c{M}\alpha$ for any $\alpha:\f{X}\to\f{X}$.

Let $\f{X}_0$ be a closed subset of $\f{X}$ containing $\r{Sp}\f{m}$ and let $\iota:\f{X}_0\to\f{X}$ denote the embedding.
Then $\c{M}\iota:\c{M}_{\f{X},\f{m}}\to\c{M}_{\f{X}_0,\f{m}}$ is surjective with kernel $I:=\{f:f|_{\f{X}^2_0}=0\}$.
Thus $\c{M}_{\f{X},\f{m}}$ is an extension of $\c{M}_{\f{X}_0,\f{m}}$ by the closed self-adjoint ideal $I$. Moreover, suppose that
$\f{X}_0$ is a retract of $\f{X}$ i.e. there is a continuous map $\rho:\f{X}\to\f{X}_0$ with $\rho\iota=\r{id}_{\f{X}_0}$.
It follows from functoriality of $\c{M}$ that $(\c{M}\iota)(\c{M}\rho)$ is the identity morphism on $\c{M}_{\f{X}_0,\f{m}}$.
This shows that the mentioned extension splits strongly in the sense of \cite[Definition 1.2]{BadeDalesLykova1}.
The discussion we just had, shows that by removing the null part of $\f{m}$ from $\f{X}$
we do not lose the principal part of the structure of $\c{M}_{\f{X},\f{m}}$. We will see
that $\r{Sp}\f{m}=\f{X}$ is a crucial condition for the study of $\c{M}_{\f{X},\f{m}}$.

For any closed subset $C$ of $\f{X}$ we have $\f{m}(C)=\inf [(1_\f{X}\ot f)\star 1_{\f{X}^2}](x,y)$,
the infimum being taken over all continuous functions $f$
on $\f{X}$ with $f(\f{X})\subseteq[0,1]$ and $f(C)=\{1\}$. Using this and inner regularity of $\f{m}$ we can find the measure
of any Borel subset. Hence we can recover $\f{m}$ from $\c{M}_{\f{X},\f{m}}$.
The author does not know if the homeomorphism type of $\f{X}$ can be recovered from $\c{M}_{\f{X},\f{m}}$.
Suppose that $\f{X}$ is finite with $\r{Sp}\f{m}=\f{X}$. It is not so hard to see that if
$\phi:\c{M}_{\f{X},\f{m}}\to\c{M}_{\f{X}',\f{m}'}$ is an isometric $*$-isomorphism then
there exist a measure preserving injective and surjective map $\alpha:\f{X}'\to\f{X}$ and a function
$\beta:\f{X}\to\mathbb{C}$, with $|\beta|=1_\f{X}$, such that
$\phi=(\c{M}\alpha)\hat{\beta}$, where $\hat{\beta}$ is defined as above.
(Note that if $\phi$ is not supposed to be isometric then this assertion is wrong.)
We suggest that this conclusion is true for any arbitrary $\f{X}$ with $\r{Sp}\f{m}=\f{X}$.

In Koopman's theory, as it is well known, the operator algebras have many applications to
study of dynamical systems and ergodic theory (\cite{EisnerFarkasHaaseNagel1}).
In this direction, the study of algebraic properties of $\c{M}_{\f{X},\f{m}}$ may be useful: Let $G$ be a discrete group of measure preserving
homeomorphisms of $\f{X}$. Then $G$ acts on $\c{M}_{\f{X},\f{m}}$ by isometric automorphisms and thus it is appropriate
to consider the crossed product Banach algebra $A:=G\ltimes\c{M}_{\f{X},\f{m}}$. It is clear that any algebraic invariant of $A$ is an invariant
of the dynamical system $(\f{X},G)$. Moreover, if the suggestion stated in the preceding paragraph is true, then $(\f{X},G)$ is
completely characterized by $A$. We plan to discuss elsewhere such possible connections with ergodic theory.
\section{Approximate units of $\c{M}$}
From now on, $\f{X}$ is a fixed compact metrizable space, $\f{m}$ is a fixed Borel probability measure on $\f{X}$
with $\r{Sp}\f{m}=\f{X}$, and $\c{M}$ will denote $\c{M}_{\f{X},\f{m}}$. We also let $\f{d}$ denote a compatible metric on $\f{X}$.
A right norm- (resp. pc-, cc-, rc-) approximate unit for $\c{M}$
is a net $(u_\lambda)_\lambda$ in $\c{M}$ such that $au_\lambda\to a$
in the norm topology (resp. pct, cct, rct) for every $a\in A$. If $\sup_{\lambda}\|u_\lambda\|_\infty<\infty$ then $(u_\lambda)_\lambda$ is called bounded. (Bounded) left and two-sided norm- (resp. pc-, cc-, rc-) approximate units are defined similarly.
It is clear that every norm-approximate unit is a pc-approximate unit.
Suppose that $x\in\f{X}$ and $\delta>0$. Throughout, $\c{O}_{x,\delta}$ denotes an open set
with $\overline{\c{B}}_{x,\delta}\subseteq\c{O}_{x,\delta}\subseteq\c{B}_{x,2\delta}$ and
$\f{m}(\c{O}_{x,\delta}\setminus\c{B}_{x,\delta})<\delta\f{m}(\c{B}_{x;\delta})$;
also $\c{E}_{x,\delta}$ denotes a continuous function from $\f{X}$ to the interval $[0,1]$ such that the restriction of
$\c{E}_{x,\delta}$ to $\overline{\c{B}}_{x,\delta}$ (resp.  $\f{X}\setminus\c{O}_{x,\delta}$) takes the constant value $1$ (resp. $0$).
\begin{theorem}\label{TT1}
There is a net in $\c{M}$ which is mutually a right cc-approximate unit and a left rc-approximate unit.
Thus the same net is also a two-sided pc-approximate unit.
\end{theorem}
\begin{proof}
The set of all pairs $(S,\epsilon)$, in which $S$ is a finite subset of $\f{X}$ and $\epsilon>0$, with the
ordering $((S,\epsilon)\leq(S',\epsilon'))\Leftrightarrow(S\subseteq S',\epsilon'\leq\epsilon)$, becomes a directed set.
For any pair $(S,\epsilon)$ choose $\delta>0$ such that $\delta<\epsilon$ and
$\c{B}_{y,2\delta}\cap\c{B}_{y',2\delta}=\emptyset$ for $y,y'\in S$ with $y\neq y'$, and let
$u_{S,\epsilon}=\sum_{y\in S}\frac{1}{\f{m}(\c{B}_{y,\delta})}\c{E}_{y,\delta}\ot\c{E}_{y,\delta}$.
We show that $(u_{S,\epsilon})_{(S,\epsilon)}$ is the desired net.
Let $f\in\c{M}$ and $r>0$ be arbitrary. Choose $\epsilon>0$ with $\epsilon<r$ such that for every $z,z',x$ if $\f{d}(z,z')<\epsilon$ then
$|f(x,z)-f(x,z')|<r$. If $x$ is arbitrary then for any pair $(S,\epsilon)$ with $y\in S$ we have
\begin{align*}
|f\star u_{S,\epsilon}-f|(x,y)
&=\frac{1}{\f{m}(\c{B}_{y,\delta})}|\int_{\c{B}_{y,\delta}}[f(x,z)-f(x,y)]\r{d}\f{m}(z)
+\int_{\c{O}_{x,\delta}\setminus\c{B}_{x,\delta}}f(x,z)\c{E}_{y,\delta}(z)\r{d}\f{m}(z)|\\
&\leq r+r\|f\|_\infty.
\end{align*}
This shows that $f\star u_{S,\epsilon}\to f$ in cct. Similarly it is proved that $u_{S,\epsilon}\star f\to f$ in rct.
\end{proof}
\begin{remark}\label{RR1}
The existence of a right (or left) pc-approximate unit for $\c{M}$ implies that $\r{Sp}\f{m}=\f{X}$.
An easy proof is as follows. Let $(u_\lambda)_{\lambda}$ be a right pc-approximate unit. Let $U$ be an arbitrary
nonempty open set in $\f{X}$ and let $f\in\b{C}(\f{X})$ be such that $f(\f{X}\setminus U)=\{0\}$ and $f(x)=1$ for some $x\in U$. Then we have
$1=(1_\f{X}\ot f)(x,x)=\lim_\lambda [(1_\f{X}\ot f)\star u_\lambda](x,x)=\lim_\lambda\int_Uf(z)u_\lambda(z,x)\mathrm{d}\f{m}(z)$.
This implies that $\f{m}(U)\neq0$. Hence $\r{Sp}\f{m}=\f{X}$.
\end{remark}
\begin{proposition}\label{PP1}
If $\c{M}$ has a bounded right (or left) pc-approximate unit then $\f{X}$ is finite.
\end{proposition}
\begin{proof}
Let $(u_\lambda)_{\lambda}$ be a right pc-approximate unit for $\c{M}$ bounded by $M>0$.
First of all we show that $\f{m}(\{x\})\neq0$ for every $x$. Assume, to get a contradiction, that $\f{m}(\{x\})=0$
for some $x$. Let $\epsilon>0$ be such that $\epsilon M<1/2$. There is an open neighborhood $U$ of $x$ with $\f{m}(U)<\epsilon$.
Let $f:\f{X}\to[0,1]$ be a continuous function with $f(x)=1$ and $f(\f{X}\setminus U)=\{0\}$. For every $\lambda$ we have
$|(1_\f{X}\ot f)\star u_\lambda|(x,x)\leq\int_U|f(z)u_\lambda(z,x)|\mathrm{d}\f{m}(z)\leq\epsilon M<1/2$.
But this is impossible because $[(1_\f{X}\ot f)\star u_\lambda](x,x)\to1$.

Now, since $\f{m}(\f{X})=1$, it is concluded that $\f{X}$ must be a countable space. Suppose that $\f{X}$ is not finite.
Then there is an infinite discrete subset $\{x_1,x_2,\cdots\}$ of $\f{X}$. For every $n$ let $f_n\in\c{M}$ be defined by
$f_n(z,z')=1$ if $z=z'=x_n$ and otherwise $f_n(z,z')=0$. Then we have
$1=f_n(x_n,x_n)=\lim_{\lambda}(f_n\star u_\lambda)(x_n,x_n)=\lim_\lambda\f{m}\{x_n\}u_\lambda(x_n,x_n)$.
It follows that $\f{m}\{x_n\}\geq 1/M$. But this contradicts $\lim_{n\to\infty}\f{m}\{x_n\}=0$. Hence, $\f{X}$ is finite.
\end{proof}
\begin{theorem}\label{TT2}
The following statements are equivalent.
\begin{enumerate}
\item[(a)] $\f{X}$ is finite.
\item[(b)] $\c{M}$ has a bounded right (or left) pc-approximate unit.
\item[(c)] $\c{M}$ has a unit.
\end{enumerate}
\end{theorem}
\begin{proof}
(b)$\Rightarrow$(a) is the statement of Proposition \ref{PP1}. (c)$\Rightarrow$(b) is trivial. (a)$\Rightarrow$(c) is easily verified
by analogy with ordinary matrix algebras.
\end{proof}
\begin{lemma}\label{LL1}
Let $x\in X$. The function $r\mapsto\f{m}(\c{B}_{x,r})$ is continuous at $r_0\in[0,\infty)$ if and only if $\f{m}\{y:\f{d}(x,y)=r_0\}=0$.
(Note that $\c{B}_{x,0}=\emptyset$.)
\end{lemma}
\begin{proof}
Straightforward.
\end{proof}
\begin{lemma}\label{LL2}
The function $x\mapsto\f{m}(\c{B}_{x,r})$ is continuous at $x_0$ if $\f{m}\{y:\f{d}(x_0,y)=r\}=0$.
\end{lemma}
\begin{proof}
For $\epsilon>0$ by Lemma \ref{LL1} there is $\delta>0$ such that $\f{m}(\c{B}_{x_0,r+\delta}\setminus\c{B}_{x_0,r-\delta})<\epsilon$.
Suppose that $y\in\c{B}_{x_0,\delta}$. Then
$\f{m}(\c{B}_{x_0,r-\delta})\leq\f{m}(\c{B}_{y,r})\leq\f{m}(\c{B}_{x_0,r+\delta})$. So $|\f{m}(\c{B}_{x_0,r})-\f{m}(\c{B}_{y,r})|<\epsilon$.
\end{proof}
\begin{lemma}\label{LL3}
Let $\delta>0$ be such that $\f{m}\{y:\f{d}(x,y)=\delta\}=0$ for every $x\in\f{X}$.
Then there exists $\delta'$ with $\delta<\delta'<2\delta$
such that $\f{m}(\c{B}_{x,\delta'}\setminus\c{B}_{x,\delta})<\delta\f{m}(\c{B}_{x,\delta})$ for every $x\in\f{X}$.
\end{lemma}
\begin{proof}
Assume, to reach a contradiction, that there is no $\delta'$ with the desired properties.
For sufficiently large $n$ we have $\delta+n^{-1}<2\delta$ and hence there
is a $x_n$ such that $\f{m}(\c{B}_{x_n,\delta+n^{-1}})-\f{m}(\c{B}_{x_n,\delta})\geq\delta\f{m}(\c{B}_{x_n,\delta})$.
Without lost of generality we can suppose that the sequence $(x_n)_n$ converges to an element $x$.
Let $r>0$ be arbitrary. For sufficiently large $n$ we have $\f{m}(\c{B}_{x_n,\delta+n^{-1}})\leq\f{m}(\c{B}_{x,\delta+r})$ and hence
$\delta\f{m}(\c{B}_{x_n,\delta})\leq\f{m}(\c{B}_{x,\delta+r})-\f{m}(\c{B}_{x_n,\delta})$. It follows from Lemma \ref{LL2} that
$\delta\f{m}(\c{B}_{x,\delta})\leq\f{m}(\c{B}_{x,\delta+r}\setminus\c{B}_{x,\delta})$. Letting $r\to0$ and using Lemma \ref{LL1}
we conclude that $\f{m}(\c{B}_{x,\delta})=0$, a contradiction.
\end{proof}
\begin{theorem}\label{TT3}
Suppose that the following condition is satisfied. \emph{(C1) $\f{X}$ has a compatible metric $\f{d}$ under which
there is a decreasing sequence $(\delta_n)_n$ of strictly positive numbers such that
$\inf_n\delta_n=0$ and $\f{m}\{y:\f{d}(x,y)=\delta_n\}=0$ for every $n$ and every $x\in\f{X}$.}
Then $\c{M}$ has a right (resp. left) norm-approximate unit. Moreover, that approximate unit can be chosen so as to be a sequence.
\end{theorem}
\begin{proof}
For every $n$ let $\delta'_n$ be such that the statement of Lemma \ref{LL3} is satisfied with $\delta,\delta'$ replaced by $\delta_n,\delta'_n$.
Let $K_n=\{(x,y):\f{d}(x,y)\leq\delta_n\}$ and $U_n=\{(x,y):\f{d}(x,y)<\delta'_n\}$. Choose a continuous function $E_n:\f{X}^2\to[0,1]$
such that $E_n(K_n)=\{1\}$ and $E_n(\f{X}^2\setminus U_n)=\{0\}$ and let $\c{E}_n$ (resp. $\c{E}'_n$) be defined by
$(x,y)\mapsto E_n(x,y)/\f{m}(\c{B}_{y,\delta_n})$ (resp. $(x,y)\mapsto E_n(x,y)/\f{m}(\c{B}_{x,\delta_n})$).
(Note that by Lemma \ref{LL2}, $\c{E}_n,\c{E}_n'\in\c{M}$.). Using Lemma \ref{LL3}, it is easily verified that
$(\c{E}_n)_n$ (resp. $(\c{E}_n')_n$) is a right (resp. left) norm-approximate unite for $\c{M}$.
\end{proof}
\begin{theorem}\label{TT4}
Suppose that the following condition is satisfied. \emph{(C2) $\f{X}$ has a compatible metric $\f{d}$ under which
there exists a sequence $(\delta_n)_n$ satisfying all properties stated in (C1) and,
in addition, $\f{m}(\c{B}_{x,\delta_n})=\f{m}(\c{B}_{y,\delta_n})$ for every $n$ and every $x,y\in\f{X}$.}
Then $\c{M}$ has a two-sided norm-approximate unit.
\end{theorem}
\begin{proof}
It is concluded from $\c{E}_n=\c{E}_n'$ where $\c{E}_n,\c{E}_n'$ are as in the proof of Theorem \ref{TT3}.
\end{proof}
\begin{example}\label{EE1}
If $\f{X}$ is the closure of a nonempty bounded open subset of $\mathbb{R}^n$ with the normalized $n$-dimensional Lebesgue measure and with
the Euclidean metric, then $\f{X}$ satisfies conditions of Theorem \ref{TT3}. More generally, if an open subset of a Riemannian manifold has compact
closure $\f{X}$ then $\f{X}$, with the geodesic distance $\f{d}$ and normalized Riemannian volume $\f{m}$, satisfies conditions of Theorem \ref{TT3}.
Indeed, $\f{m}\{y:\f{d}(x,y)=r\}=0$ for every $r$ and $x$.
\end{example}
\begin{example}\label{EE2}
Any closed Riemannian manifold $\f{X}$ which has constant (positive) sectional curvature
(e.g. standard spheres and tori, compact Lie groups with invariant Riemannian metrics),
with geodesic distance $\f{d}$ and normalized Riemannian volume $\f{m}$, satisfies conditions of Theorem \ref{TT4}.
Indeed, in addition to the property mentioned in Example \ref{EE1}, we have $\f{m}(\c{B}_{x,r})=\f{m}(\c{B}_{y,r})$ for every $r,x,y$.
\end{example}
\begin{example}\label{EE3}
Let $\f{X}$ be a second countable compact Hausdorff group. It is well-known that $\f{X}$ has a compatible bi-invariant metric $\f{d}$
i.e. $\f{d}(zxz',zyz')=\f{d}(x,y)$ for every $x,y,z,z'\in\f{X}$ (see \cite{Klee1} or \cite[Corollary A4.19]{HofmannMorris1}).
We show that $\f{d}$ with the normalized Haar measure $\f{m}$ satisfies (C1) and hence (because of invariant property of $\f{m}$) satisfies (C2):
Suppose, on the contrary, that there is no sequence $(\delta_n)_n$ satisfying (C1) for $\f{d}$. So there must be $\epsilon>0$ such that
$\f{m}\{y:\f{d}(e,y)=r\}\neq 0$ for every nonzero $r<\epsilon$; thus $\f{m}(\c{B}_{e,\epsilon})=\infty$, a contradiction.
\end{example}
\section{The center of $\c{M}$}
It is clear that if $\f{X}$ is finite then the center of $\c{M}$ is the one-dimensional subalgebra of scaler multiples of the unit of $\c{M}$.
But in the infinite case the situation is different:
\begin{theorem}\label{TT5}
If $\f{X}$ is infinite then the center of $\c{M}$ is zero.
\end{theorem}
\begin{proof}
Suppose that $f$ is in the center of $\c{M}$. Let $x,y$ be arbitrary in $\f{X}$ with $x\neq y$, and $\delta>0$ be such that $\f{d}(x,y)>4\delta$.
Let $g:=\frac{1}{\f{m}(\c{B}_{x,\delta})}\c{E}_{x,\delta}\ot\c{E}_{x,\delta}$ and
$h_\delta:=\frac{1}{\f{m}(\c{B}_{x,\delta})}\c{E}_{x,\delta}\ot\c{E}_{y,\delta}$.
Then we have $f\star g(x,y)=0$ and hence $g\star f(x,y)=0$. We have,
$$|f|(x,y)=|g\star f-f|(x,y)\leq\frac{1}{\f{m}(\c{B}_{x,\delta})}\int_{\c{B}_{x,\delta}}|f(z,y)-f(x,y)|\r{d}\f{m}(z)+\delta\|f\|_\infty.$$
By this inequality and continuity of $f$ we conclude that $f(x,y)=0$. Also, a simple computation shows that
$\lim_{\delta\to0} f\star h_\delta(x,y)=f(x,x)$ and $\lim_{\delta\to0} h_\delta\star f(x,y)=f(y,y)$. Thus we have $f(x,x)=f(y,y)$.
Now, suppose that $\f{X}$ is infinite. Then there is a sequence $(x_n)_{n\geq0}$ such that $x_n\to x_0$ and $x_0\neq x_n$
for every $n\geq1$. Thus $f(x_0,x_0)=\lim_{n\to\infty}f(x_0,x_n)=0$ and hence $f(x,x)=f(x_0,x_0)=0$. This completes the proof.
\end{proof}
\section{The ideal structure of $\c{M}$}
It is clear that the involution $*$ induces a one-to-one correspondence between norm- (resp. rc-, cc-, pc-) closed right ideals and
norm- (resp. cc-, rc-, pc-) closed left ideals of $\c{M}$. Also any self-adjoint right or left ideal is a two-sided ideal.
The rc-closure of any right ideal is a right ideal and the cc-closure of any left ideal is a left ideal.
For any norm-closed linear subspace ${V}$ of $\b{C}(\f{X})$ we let
$\c{R}_{V}:=\{f\in\c{M}:f({\cdot},y)\in V\}$ and $\c{L}_{V}:=\{f\in\c{M}:f(x,{\cdot})\in V\}$.
It is clear that $\c{R}_{V}^*=\c{L}_{\bar{V}}$ and $\c{L}_{V}^*=\c{R}_{\bar{V}}$ where $\bar{V}:=\{\bar{f}:f\in V\}$.
\begin{theorem}\label{TT6}
$\c{R}_{V}$ (resp. $\c{L}_{V}$) is a cc-closed right (resp. rc-closed left) ideal in $\c{M}$. Moreover, if $V$ is pc-closed
then $\c{R}_V$ (resp. $\c{L}_{V}$) is pc-closed.
\end{theorem}
\begin{proof}
It is clear that $\c{R}_{V}$ is a cc-closed linear subspace of $\c{M}$. Let $f\in\c{R}_{V}$ and $g\in\c{M}$. For every
$y$ let $h_y:\f{X}\to{V}$ be defined by $h_y(z)=f({\cdot},z)g(z,y)$. Then the Bochner integral $\int_\f{X}h_y\r{d}\f{m}$ exists and belongs to
$V$ (\cite[Proposition 1.31]{HytonenNeervenVeraarWeis1}). Since $f\star g({\cdot},y)=\int_\f{X}h_y\r{d}\f{m}$, we have $f\star g\in\c{R}_V$.
Thus $\c{R}_V$ is a right ideal. Also, $\c{L}_{V}=\c{R}_{\bar{{V}}}^*$ is a rc-closed left ideal. The second part of the theorem
is trivial.
\end{proof}
\begin{theorem}\label{TT7}
Let ${R}$ be a norm-closed right ideal of $\c{M}$ and let ${V}=\{f({\cdot},y):f\in{R}, y\in\f{X}\}$.
Then ${V}$ is a norm-closed linear subspace of $\b{C}(\f{X})$ and the cc-closure of $R$ is equal to $\c{R}_{V}$.
Moreover, if $R$ is pc-closed then $V$ is pc-closed.
\end{theorem}
\begin{proof}
Suppose that $f\in{R}$ and $y\in\f{X}$. Let $\epsilon>0$ be arbitrary and $\delta>0$ with $\delta<\epsilon$ be such that if
$\f{d}(z,z')<\delta$ then $|f(x,z)-f(x,z')|<\epsilon$ for every $x$. Then for every $x,y'$ we have
$|\frac{1}{\f{m}(\c{B}_{y,\delta})}[f\star(\c{E}_{y,\delta}\ot1)](x,y')-f(x,y)|\leq \epsilon+\epsilon\|f\|_\infty$.
This implies that there exists $F_{f,y}\in{R}$ with $F_{f,y}(x,z)=f(x,y)$ for every $x,z$.
Let $h,h'\in{V}$. Let $f,f'\in{R}$ and $y,y'\in\f{X}$ be such that $h=f({\cdot},y)$ and $h'=f'({\cdot},y')$. We have
$h+h'=[F_{f,y}+F_{f',y'}]({\cdot},z)$ for any arbitrary $z$ and thus $h+h'\in{V}$. This shows that $V$ is a linear subspace.
Suppose that $g\in\b{C}(\f{X})$ is a limit point of ${V}$.
There are sequences $(f_{n})_n\in{R}$ and $(y_n)_n\in\f{X}$ such that $f_n({\cdot},y_n)\to g$. It is clear that the sequence
$(F_{f_n,y_n})_n\in R$ converges to an element $G$ of ${R}$ with $G({\cdot},z)=g$ for every $z$. This shows that ${V}$ is norm-closed.
(A similar argument shows that if $R$ is pc-closed then $V$ is pc-closed.) To complete the proof, it is enough to show that
if $K\in\c{R}_{V}$ then there exists a net in $R$ converging to $K$ in cct. Let $K\in\c{R}_{V}$ be fixed. For every $y$ there are
$k_y\in R$ and $\alpha(y)\in\f{X}$ such that $K({\cdot},y)=k_y({\cdot},\alpha(y))$. For every $\epsilon>0$ and every finite subset $S$ of $\f{X}$
there exists $\delta>0$ with the following three properties.
\begin{enumerate}
\item[--] $\delta\|k_{y}\|_\infty<\epsilon/2$ for every $y\in S$.
\item[--] $\c{B}_{y,2\delta}\cap\c{B}_{y',2\delta}=\emptyset$ for $y,y'\in S$ with $y\neq y'$.
\item[--] If $\f{d}(z,z')<2\delta$ then $|k_{y}(x,z)-k_{y}(x,z')|<\epsilon/2$ for every $y\in S$.
\end{enumerate}
Let $K_{S,\epsilon}:=\sum_{y\in S}\frac{1}{\f{m}(\c{B}_{\alpha(y),\delta})}h_{y}\star(\c{E}_{\alpha(y),\delta}\ot\c{E}_{y,\delta})\in{R}$.
Then $\|K_{S,\epsilon}({\cdot},y)-G({\cdot},y)\|_\infty<\epsilon$ for every $y\in S$. Considering the set of all pairs
$(S,\epsilon)$ as a directed set in the obvious way, shows that $K_{S,\epsilon}\xrightarrow{cct}K$.
\end{proof}
Passing through the involution and using Theorem \ref{TT7}, we conclude that for any
norm-closed left ideal $L$ of $\c{M}$, ${V}:=\{f(x,{\cdot}):f\in{L}, x\in\f{X}\}$ is a norm-closed linear subspace and rc-closure of $L$
is equal to $\c{L}_{V}$. Moreover, if $L$ is pc-closed then $V$ is pc-closed.
\begin{corollary}\label{CC1}
The mapping ${V}\mapsto\c{R}_{V}$ (resp. ${V}\mapsto\c{L}_{V}$) establishes a 1-1 correspondence
between norm-closed linear subspaces of $\b{C}(\f{X})$ and cc-closed right (resp. rc-closed left) ideals of $\c{M}$, and also
between pc-closed linear subspaces of $\b{C}(\f{X})$ and pc-closed right (resp. left) ideals of $\c{M}$.
In particular, 1-dimensional and norm-closed 1-codimensional subspaces of $\b{C}(\f{X})$ correspond respectively to minimal and maximal
cc-closed right (resp. rc-closed left) ideals of $\c{M}$.
\end{corollary}
\begin{corollary}\label{CC2}
There is no nontrivial ideal in $\c{M}$ mutually closed under both cct and rct.
In particular, there is no nontrivial pc-closed ideal in $\c{M}$.
\end{corollary}
\begin{proof}
Let ${I}$ be a nonzero cc-closed and rc-closed ideal. There are closed linear subspaces
${V},{W}\subseteq\b{C}(\f{X})$ such that ${I}=\c{R}_{V}=\c{L}_{W}$. Since ${V}\neq0$ there are $f_0\in{V}$ and $x_0\in\f{X}$ with $f_0(x_0)=1$.
For every $g\in\b{C}(\f{X})$ we have $f_0\ot g\in\c{R}_{V}$. Thus $g=(f_0\ot g)(x_0,-)\in{W}$ and ${W}=\b{C}(\f{X})$. So, ${I}=\c{M}$.
\end{proof}
\section{Canonical representations of $\c{M}$}
For a Banach algebra $A$ a Banach space $E$ is called Banach left $A$-module if $E$ is a left $A$-module in the algebraic sense and such that
the action of $A$ on $E$ is a bounded bilinear operator. Banach right $A$-modules and Banach $A$-bimodules are defined similarly.
Let $\b{B}(E)$ denote the Banach algebra of bounded linear operators on $E$ and $\b{K}(E)\subseteq \b{B}(E)$ be the closed ideal of compact operators.
Any Banach left $A$-module structure on $E$ gives rise to a bounded representation $A\to\b{B}(E)$, $a\mapsto(\omega\mapsto a\omega)$, and vice versa.
The statements of the following theorem are standard results and can be find for instance in \cite{HalmosSunder1}.
\begin{theorem}\label{TT8}
Let $E$ denote any of the Banach spaces $\b{L}^p(\f{m})$ ($1\leq p\leq\infty$) or $\b{C}(\f{X})$. Then $\rho:\c{M}\to\b{K}(E)$, defined by
$[\rho(f)g](x)=\int_\f{X}f(x,y)g(y)\r{d}\f{m}(y)$ ($g\in E$),
is a well-defined faithful bounded representation. Moreover, the following statements hold.
\begin{enumerate}
\item[(i)] In the case that $E=\b{L}^2(\f{m})$, $\rho$ is a $*$-representation.
\item[(ii)] In the case that $E=\b{L}^\infty(\f{m})$ or $E=\b{C}(\f{X})$, $\rho$ is isometric.
\end{enumerate}
\end{theorem}
It is clear that for any Banach space $E$, $\b{C}(\f{X};E)$ is a Banach right (resp. left) $\c{M}$-module in the canonical way.
Its module action is denoted by the same symbol $\star$ and is given by $(g\star f)(y)=\int_\f{X}g(z)f(z,y)\r{d}\f{m}(z)$
(resp. $(f\star g)(x)=\int_\f{X}f(x,z)g(z)\r{d}\f{m}(z)$) for $f\in\c{M}$ and $g\in\b{C}(\f{X};E)$. Similarly, $\b{C}(\f{X}^2;E)$ becomes a Banach
$\c{M}$-bimodule.
\section{Derivations on $\c{M}$}
Let $A$ be a Banach algebra and $E$ be a Banach $A$-bimodule. A (bounded) \emph{derivation} from $A$ to $E$ is a (bounded) linear map
$D:A\to E$ satisfying $D(ab)=aD(b)+D(a)b$ ($a,b\in A$). $D$ is called \emph{inner} if there exists $\omega\in E$ such that $D(a)=a\omega-\omega a$
for every $a$. $D$ is called \emph{approximately inner} \cite{GhahramaniLoy1} if there is a net $(\omega_\lambda)_\lambda$ in $E$ such that
$D(a)=\lim_\lambda a\omega_\lambda-\omega_\lambda a$. If $(\omega_\lambda)_\lambda$ can be chosen so as to be a sequence
then $D$ is called \emph{sequentially} approximate inner.
\begin{theorem}\label{TT9}
Suppose that the condition (C2) of Theorem \ref{TT4} is satisfied, and let $E$ be a Banach $\c{M}$-bimodule
such that its module operation $\diamond:\c{M}\ot E\ot\c{M}\to E$ is continuous w.r.t. injective tensor norm, and such that for every norm approximate
unit $(\c{E}_n)_n$ of $\c{M}$ we have $\c{E}_n\diamond \omega\to \omega$ for every $\omega\in E$.
Then any bounded derivation from $\c{M}$ to $E$ is sequentially approximate inner.
\end{theorem}
\begin{proof}
Let $D:\c{M}\to E$ be a bounded derivation.
Let $\Gamma:\c{M}\check{\ot}\c{M}\to E$ be the bounded linear map defined by $f\ot g\mapsto f\diamond D(g)$.
Also let $\Lambda:\c{M}\check{\ot}\c{M}\to\c{M}$ denote the convolution product.
It is not hard to verify the following two identities
for $h\in\c{M}$ and $F\in\c{M}\check{\ot}\c{M}$.
$$\Gamma(h\star F)=h\diamond\Gamma(F),\hspace{10mm}\Gamma(F\star h)=\Lambda(F)\diamond D(h)+\Gamma(F)\diamond h.$$
Let the sequence $(\delta_n)_n$ be as in the statement of Theorem \ref{TT4} and let $\alpha_n=\f{m}(\c{B}_{x,\delta_n})$ for every $x\in\f{X}$.
By Lemma \ref{LL3} there is $r_n$ such that $\delta_n<r_n<2\delta_n$ and $\f{m}(\c{B}_{x,r_n}-\c{B}_{x,\delta_n})<\delta_n\alpha_n$.
Choose a continuous function $G_n:\f{X}^2\to[0,1]$ such that $G_n$ has constant values $1$ and $0$ respectively on
$\{(x,y):\f{d}(x,y)\leq\delta_n\}$ and $\{(x,y):\f{d}(x,y)\geq r_n\}$, and let $\c{G}_n\in\b{C}(\f{X}^4)$ be defined by
$\c{G}_n(x,z,z',y)=\frac{1}{\alpha_n}G_\delta(x,y)$. Note that we have $\Lambda(\c{G}_n)=\frac{1}{\alpha_n}G_n$.
It is not hard to verify that $(\Lambda(\c{G}_n))_n$ is a two-sided norm-approximate unit for $\c{M}$ and
$\lim_{n\to\infty}f\star\c{G}_n-\c{G}_n\star f=0$ for every $f\in\c{M}$. Let $K_n=\Gamma(\c{G}_n)\in E$. For the sequence $(K_n)_n$ and $h\in\c{M}$
we have,
\begin{align*}
\lim_{n\to\infty}h\diamond K_n-K_n\diamond h&=\lim_{n\to\infty}h\diamond \Gamma(\c{G}_n)-\Gamma(\c{G}_n)\diamond h\\
&=\lim_{n\to\infty}\Gamma(h\star\c{G}_n)-\Gamma(\c{G}_n\star h)+\Lambda(\c{G}_n)\diamond D(h)\\
&=\Gamma(\lim_{n\to\infty}h\star\c{G}_n-\c{G}_n\star h)+D(h)\\
&=D(h).
\end{align*}
This completes the proof.
\end{proof}
For any Banach space $E$, the Banach $\c{M}$-bimodule $\b{C}(\f{X}^2;E)$,
mentioned in the preceding section, satisfies the conditions of Theorem \ref{TT9}.
\bibliographystyle{amsplain}

\end{document}